\documentclass[reqno,11pt]{amsart}

\usepackage{xcolor}
\usepackage{graphicx}
\usepackage[hidelinks]{hyperref}
\usepackage{amscd}
\usepackage{amsmath}
\usepackage{amsthm}
\usepackage{amssymb}
\usepackage{latexsym,array}
\usepackage{amsfonts}
\usepackage{shadow}
\usepackage{amsbsy}
\usepackage{amssymb}
\usepackage{physics}
\usepackage{dsfont}
\usepackage{mathtools}
\usepackage{enumitem}
\usepackage{doi}

\usepackage[notref, notcite, final]{showkeys}
\usepackage{color}
\usepackage{graphicx}
\usepackage[hidelinks]{hyperref}
\usepackage{cleveref}
\usepackage{fullpage}
\usepackage{comment}
\usepackage{tikz-cd}
\usepackage{tikz}
\usepackage{float}

\usepackage{dsfont}

\usepackage{cleveref}

\numberwithin{equation}{section}

\newtheorem{theorem}{Theorem}[section]

\newtheorem{corollary}[theorem]{Corollary}
\newtheorem{proposition}[theorem]{Proposition}

\theoremstyle{definition}
\newtheorem{remark}[theorem]{{\bf Remark}}
\newtheorem{definition}[theorem]{Definition}

\newcommand{\R}{\mathbb{R}}

\newcommand{\boundOP}{\mathcal{B}}

\newcommand{\ov}{\overline}

\DeclareMathOperator\dom{dom}

\title[]{Slice hyperholomorphicity of the $S$-resolvent operators and boundary conditions}

\author[F. Mantovani]{Francesco Mantovani*}
\address{(FM) Politecnico di Milano, Dipartimento di Matematica, Via E. Bonardi, 9, 20133 Milano, Italy}
\email{francesco.mantovani@polimi.it}
\thanks{*Corresponding author}

\date{}

\begin{document}
	
	\maketitle
	
\begin{abstract}
The foundation of spectral theory on the $S$-spectrum can be traced back
to the quaternionic framework of quantum mechanics.
The concept of $S$-spectrum for quaternionic operators emerged as the natural spectrum
in slice hyperholomorphic functional calculi, known as the $S$-functional calculus
and also utilized in the quaternionic spectral theorem.
This spectral theory extends to Clifford operators.
A key distinction from classical complex spectral theory lies in the definition of the $S$-spectrum,
 which is second order in the operator $T$, and in the $S$-resolvent operators that turns out
 to be the product of two different operators.
 This study delves into the analyticity of the $S$-resolvent operators
 under specified boundary conditions for the $S$-spectral problem.
 The spectral theory on the $S$-spectrum also provides deeper insights into classical spectral theory.
\end{abstract}



\tableofcontents

\section{Introduction}\label{INTRO}

Vector operators appear in several scientific contexts and can be interpreted as Clifford right-linear operators on Banach modules, with the following identification
\begin{equation*}
    T = (T_0,\dots,T_{\abs{\mathcal{A}}}) \longleftrightarrow T = \sum_{A \in \mathcal{A}} e_A T_A
\end{equation*}
where $e_A$ runs over all the possible ordered products of the imaginary units $e_1,\dots,e_n$ which generate the Clifford algebra $\R_n$. \\
In a Banach module $V$ over the Clifford algebra $\R_n$, the spectral theory for a right-linear operator $T: \dom(T) \subseteq V \to V$ differs from the classic complex theory because the spectral parameter $s \in \R^{n+1}$ does not commute with the operator $T$, hence the notion of spectrum is no longer associated with the invertibility of the operator $T-s$, but involves the invertibility of the second order operator
\begin{equation*}
    Q_s[T]:= T^2-2s_0T+\abs{s}^2 I,\quad \dom(Q_s[T])=\dom(T^2).
\end{equation*}
This leads to the definition of $S$-resolvent set and of the $S$-spectrum:
\begin{equation*}
    \rho_S(T) = \{ s \in \R^{n+1}\,:\,Q_s[T]^{-1}\in \boundOP(V)\}\quad\text{and}\quad \sigma_S(T)=\R^{n+1}\setminus\rho_S(T),
\end{equation*}
where $\boundOP(V)$ denotes the set of bounded right-linear operators.
In the introduction of the book \cite{CGK} one finds a description on the discovery of the $S$-spectrum.

 For every $s \in \rho_S(T)$, the inverse $Q_s[T]^{-1}$ is called the pseudo-resolvent operator of $T$ at $s$. However, the second main difference with the complex case is that the operator valued function
\begin{equation*}
    \rho_S(T) \ni s \longmapsto Q_s[T]^{-1}\in \boundOP(V),
\end{equation*}
is not slice hyperholomorphic (namely the notion of holomorphicity which is needed in this noncommutative setting), but the role of the resolvent operator is played by the left and the right $S$-resolvent operators
\begin{equation*}
    S^{-1}_L(s,T)=Q_s[T]^{-1}\ov{s}-TQ_{s}[T]^{-1}\quad \text{and}\quad S^{-1}_R(s,T)=(\ov{s}-T)Q_{s}[T]^{-1},\quad  s \in \rho_S(T),
\end{equation*}
which are, respectively, a right and a left slice hyperholomorphic function on $\rho_S(T)$. These resolvent operators are then used in the definition of the $S$-functional calculus for Clifford operators, and other functional calculi based on the concept of slice hyperholomorphicity. Other details about the spectral theory on the $S$-spectrum can be found in \cite{ACS2016,AlpayColSab2020,FJBOOK,CGK,ColomboSabadiniStruppa2011}, where one can find also the development of the theory of slice hyperholomorphic functions.

\medskip

Although our theory is valid in Banach modules, we restrict our discussion of boundary conditions to the Hilbert modules $H^1(\Omega)$, as explicit results are available in the literature.
One of the main examples of Clifford operator is the so-called gradient operator with non-constant coefficients
\begin{equation*}
    T = \sum_{i=1}^n e_i a_i(x) \frac{\partial}{\partial x_i}, \quad \dom(T) = H^1(\Omega) \subset L^2(\Omega)
\end{equation*}
where the coefficients $a_1,\dots,a_n$ are defined on some minimally smooth domain $\Omega\subseteq\mathbb{R}^n$, in $n\geq 3$ dimensions. This operator represents various physical laws, such as Fourier's law for heat propagation and Fick's law for mass transfer diffusion and was studied in \cite{CMS24} under various boundary conditions. 
Since the notion of spectrum is associated with the invertibility of the second order operator $Q_s[T]$, depending on the spectral parameter $s$, it is natural to consider
the spectral problem 
\begin{equation*}
    \begin{cases}
        T^2u-2s_0Tu+\abs{s}^2u = f;\\
        u\text{ satisfies }(B);
    \end{cases}
\end{equation*}
where $(B)$ represents some boundary conditions, such as Dirichlet or Robin-like, 
considering for example $u \in \dom(T^2)$ and $f \in L^2(\Omega)$. However, if we imposed the boundary conditions on the level of the domain of the first-order differential operator, for example
\begin{equation*}
    \dom(T) = \{ u \in H^1(\Omega)\,:\, u\text{ satisfies }(B)\},
\end{equation*}
then we would get
\begin{align*}
    \dom(T^2)&= \{u \in \dom(T) \,:\, Tu \in \dom(T)\}
    \\
    &= \{ u \in H^1(\Omega)\,:\, Tu \in H^1(\Omega)\text{ and both $Tu$ and $u$ satisfy  }(B) \},
\end{align*}
and the condition $s \in \rho_S(T)$ would correspond to the less natural problem
\begin{equation*}
    \begin{cases}
        T^2u-2s_0Tu+\abs{s}^2u = f;\\
        u\text{ satisfies }(B);\\
         Tu\text{ satisfies }(B).
    \end{cases}
\end{equation*}

 In this paper we explore the effects on the notion of $S$-spectrum if we detach the operator $T$ and the pseudo-resolvent $Q_s[T]$, namely if we study the invertibility of $Q_s[T]$ not over the domain $\dom(T^2)$, which naturally comes from the definition of $T$, but only over a subspace of it. This would allow to include in the domain of $T$ only the conditions which concern the regularity of the functions, while we study the invertibility of $Q_s[T]$ with some extra conditions which regard the values of the functions on the boundary of the domain $\Omega$. More precisely, given a closed operator $T$ with domain $\dom(T)$, we introduce the subspace $B \subseteq \dom(T)$, which represents the subspace of the functions which satisfy the given boundary conditions, and define $Q_{s,B}[T]$ as the restriction of $Q_{s}[T]$ to the intersection of $\dom(T^2)$ and $B$. The $S$-spectrum of $T$ with boundary conditions $B$ will be the set of paravectors $s \in \R^{n+1}$ such that $Q_{s,B}[T]$ has not a bounded inverse. Using this new definition, we find a notion of spectrum which shares some properties with the classic one, except the holomorphicity of the associated resolvent operators. Indeed, we will see that the holomorphicity of the $S$-resolvent operators is connected with the possibility of commuting $T$ and the inverse of $Q_{s,B}[T]$. This is always possible when we work without boundary conditions (that is, $B=\dom(T)$), but when $B$ is not trivial, then the problem of identifying the points where these operators commute becomes more complicated. The main result that we will be able to prove is that, if $s\notin \R$, the vectors $v\in V$ where the Cauchy-Riemann equations for the $S$-resolvent operators hold are exactly the ones such that
 \begin{equation*}
     TQ_{s,B}[T]^{-1} Q_{s,B}[T]^{-1}v = Q_{s,B}[T]^{-1}TQ_{s,B}[T]^{-1}v,
 \end{equation*}
that is $T$ and $Q_{s,B}[T]^{-1}$ commute on $Q_{s,B}[T]^{-1}$. The situation is slightly different for real points, indeed the vectors $v \in V$ where the Cauchy-Riemann equations are satisfied are precisely the ones such that
\begin{equation*}
    T [T,Q_{s,B}[T]^{-1}]Q_{s,B}[T]^{-1}v = s[T,Q_{s,B}[T]^{-1}]Q_{s,B}[T]^{-1}v,
\end{equation*}
where
$$
[T,Q_{s,B}[T]^{-1}]:=TQ_{s,B}[T]^{-1}-Q_{s,B}[T]^{-1}T,
$$
that is, $[T,Q_{s,B}[T]]^{-1}Q_{s,B}[T]^{-1}v$ is an eigenvector (possibly zero) of $T$ with respect to the eigenvalue $s$.

The gradient operator with nonconstant coefficients is not the only one for which this notion of spectrum could be useful. Indeed, in \cite{DIRACHYPSPHE} the authors studied the Dirac operator in hyperbolic and spherical spaces with Dirichlet and Robin-like boundary conditions. More general, this theory can be applied to Dirac operators on manifolds with different boundary conditions.

\medskip
The paper is organized as follows: Section \ref{INTRO} provides the introduction. Section \ref{PREL} presents the necessary preliminaries regarding Clifford algebras and Clifford modules. In Section \ref{SBOUND}, we examine the S-spectrum with boundary conditions. Section \ref{CONTIUN} is devoted to the continuity of the S-resolvent operators with boundary conditions, while Section \ref{section:analyticity} establishes the analyticity of these operators.

\section{Preliminaries on Clifford Algebras and Clifford modules}\label{PREL}

In this section we will fix the algebraic and functional analytic setting of this paper. The underlying algebra we will consider in this article will be the real {\it Clifford algebra} $\mathbb{R}_n$ over $n$ {\it imaginary units} $e_1,\dots,e_n$, which satisfy the relations
\begin{equation*}
e_i^2=-1\qquad\text{and}\qquad e_ie_j=-e_je_i,\qquad i\neq j\in\{1,\dots,n\}.
\end{equation*}
More precisely, $\mathbb{R}_n$ is given by
\begin{equation*}
\mathbb{R}_n:=\Big\{\sum\nolimits_{A\in\mathcal{A}}s_Ae_A\;\Big|\;s_A\in\mathbb{R},\,A\in\mathcal{A}\Big\},
\end{equation*}
using the index set
\begin{equation*}
\mathcal{A}:=\big\{(i_1,\dots,i_r)\;\big|\;r\in\{0,\dots,n\},\,1\leq i_1<\dots<i_r\leq n\big\},
\end{equation*}
and the {\it basis vectors} $e_A:=e_{i_1}\dots e_{i_r}$. Note that for $A=\emptyset$ the empty product of imaginary units is the real number $e_\emptyset=1$. Furthermore, we define for every Clifford number $s\in\mathbb{R}_n$ its {\it conjugate} and its {\it absolute value}
\begin{equation}\label{Eq_Conjugate_Norm}
\overline{s}:=\sum\nolimits_{A\in\mathcal{A}}(-1)^{\frac{|A|(|A|+1)}{2}}s_Ae_A\qquad\text{and}\qquad|s|^2:=\sum\nolimits_{A\in\mathcal{A}}s_A^2.
\end{equation}
An important subset of the Clifford numbers are the so called {\it paravectors}
\begin{equation*}
\mathbb{R}^{n+1}:=\big\{s_0+s_1e_1+\dots+s_ne_n\;\big|\;s_0,s_1,\dots,s_n\in\mathbb{R}\big\}.
\end{equation*}
For any paravector $s\in\mathbb{R}^{n+1}$, we define the {\it imaginary part}
\begin{equation*}
\Im(s):=s_1e_1+\dots+s_ne_n,
\end{equation*}
and the conjugate and the modulus in \eqref{Eq_Conjugate_Norm} reduce to
\begin{equation*}
\overline{s}=s_0-s_1e_1-\dots-s_ne_n\qquad\text{and}\qquad|s|^2=s_0^2+s_1^2+\dots+s_n^2.
\end{equation*}
The sphere of purely imaginary paravectors with modulus $1$, is defined by
\begin{equation}\label{Eq_S}
\mathbb{S}:=\big\{J\in\mathbb{R}^{n+1}\;\big|\;J_0=0,\,|J|=1\big\}.
\end{equation}
Any element $J\in\mathbb{S}$ satisfies $J^2=-1$ and hence the corresponding hyperplane
\begin{equation*}
\mathbb{C}_J:=\big\{x+Jy\;\big|\;x,y\in\mathbb{R}\big\}
\end{equation*}
is an isomorphic copy of the complex numbers. Moreover, for every paravector $s\in\mathbb{R}^{n+1}$ we consider the corresponding {\it $(n-1)$--sphere}
\begin{equation*}
[s]:=\big\{s_0+J|\Im(s)|\;\big|\;J\in\mathbb{S}\big\}.
\end{equation*}
A subset $U\subseteq\mathbb{R}^{n+1}$ is called {\it axially symmetric}, if $[s]\subseteq U$ for every $s\in U$. \medskip

Next, we introduce the notion of slice hyperholomorphic functions $f:U\rightarrow\mathbb{R}_n$, defined on an axially symmetric open set $U\subseteq\mathbb{R}^{n+1}$.

\begin{definition}[Slice hyperholomorphic functions]\label{defi_Slice_hyperholomorphic_functions}
Let $U\subseteq\mathbb{R}^{n+1}$ be open, axially symmetric and consider
\begin{equation*}
\mathcal{U}:=\big\{(x,y)\in\mathbb{R}^2\;\big|\;x+\mathbb{S}y\subseteq U\big\}.
\end{equation*}
A function $f:U\rightarrow\mathbb{R}_n$ is called {\it left} (resp. {\it right}) {\it slice hyperholomorphic}, if there exist continuously differentiable functions $f_0,f_1:\mathcal{U}\rightarrow\mathbb{R}_n$, such that for every $(x,y)\in\mathcal{U}$:

\begin{enumerate}
\item[i)] The function $f$ admits for every $J\in\mathbb{S}$ the representation
\begin{equation}\label{Eq_Holomorphic_decomposition}
f(x+Jy)=f_0(x,y)+Jf_1(x,y),\quad\Big(\text{resp.}\;f(x+Jy)=f_0(x,y)+f_1(x,y)J\Big).
\end{equation}

\item[ii)] The functions $f_0,f_1$ satisfy the {\it compatibility conditions}
\begin{equation}\label{Eq_Symmetry_condition}
f_0(x,-y)=f_0(x,y)\qquad\text{and}\qquad f_1(x,-y)=-f_1(x,y).
\end{equation}

\item[iii)] The functions $f_0,f_1$ satisfy the {\it Cauchy-Riemann equations}
\begin{equation}\label{Eq_Cauchy_Riemann_equations}
\frac{\partial}{\partial x}f_0(x,y)=\frac{\partial}{\partial y}f_1(x,y)\qquad\text{and}\qquad\frac{\partial}{\partial y}f_0(x,y)=-\frac{\partial}{\partial x}f_1(x,y).
\end{equation}
\end{enumerate}

The class of left (resp. right) slice hyperholomorphic functions on $U$ is denoted by $\mathcal{SH}_L(U)$ (resp. $\mathcal{SH}_R(U)$). In the special case that $f_0$ and $f_1$ are real valued, we call $f$ {\it intrinsic} and the space of intrinsic functions is denoted by $\mathcal{N}(U)$.
\end{definition}

Next, we turn our attention to Banach modules over $\mathbb{R}_n$. For a real Banach space $V_\mathbb{R}$ with corresponding norm $\Vert\cdot\Vert_\mathbb{R}^2=\langle\cdot,\cdot\rangle_\mathbb{R}$, we define the corresponding Banach module
\begin{equation*}
V:=\Big\{\sum\nolimits_{A\in\mathcal{A}}v_A\otimes e_A\;\Big|\;v_A\in V_\mathbb{R},\,A\in\mathcal{A}\Big\},
\end{equation*}
and equip it with the norm
\begin{equation*}
\Vert v\Vert^2:=\sum\nolimits_{A\in\mathcal{A}}\Vert v_A\Vert_\mathbb{R}^2,\qquad v\in V.
\end{equation*}
For any vector $v=\sum_{A\in\mathcal{A}}v_A\otimes e_A\in V$ and any Clifford number $s=\sum_{B\in\mathcal{A}}s_Be_B\in\mathbb{R}_n$, we establish the left and the right scalar multiplication
\begin{align*}
sv:=&\sum\nolimits_{A,B\in\mathcal{A}}(s_Bv_A)\otimes(e_Be_A), && \textit{(left-multiplication)} \\
vs:=&\sum\nolimits_{A,B\in\mathcal{A}}(v_As_B)\otimes(e_Ae_B). && \textit{(right-multiplication)}
\end{align*}
From \cite[Lemma 2.1]{CMS24} we recall the well known properties of these products
\begin{subequations}
\begin{align}
\Vert sv\Vert&\leq 2^{\frac{n}{2}}|s|\Vert v\Vert\qquad\text{and}\qquad\Vert vs\Vert\leq 2^{\frac{n}{2}}|s|\Vert v\Vert,\qquad\text{if }s\in\mathbb{R}_n, \label{Eq_Norm_estimate_Rn} \\
\Vert sv\Vert&=\Vert vs\Vert=|s|\Vert v\Vert,\hspace{4.83cm}\text{if }s\in\mathbb{R}^{n+1}. \label{Eq_Norm_estimate_paravector}
\end{align}
\end{subequations}

For any Banach module $V$ we will denote the set of {\it bounded} right-linear operators
\begin{equation*}
\mathcal{B}(V):=\big\{T:V\rightarrow V\text{ right-linear}\;\big|\;\text{$T$ is bounded}\big\},
\end{equation*}
as well as the space of {\it closed operators}
\begin{equation*}
\mathcal{K}(V):=\big\{T:V\rightarrow V\text{ right-linear}\;\big|\;\dom(T)\subseteq V\text{ is right linear},\,T\text{ is closed}\big\}.
\end{equation*}

In difference to complex Banach spaces, in Banach modules over $\R_n$ the spectrum of a closed operator $T\in\mathcal{K}(V)$ is connected to the bounded invertibility of the operator
\begin{equation*}
Q_s[T]:=T^2-2s_0T+|s|^2 I,\qquad\text{with }\dom(Q_s[T])=\dom(T^2).
\end{equation*}
This suggests the following definition of $S$-spectrum.

\begin{definition}[$S$-Spectrum]
For every $T\in\mathcal{K}(V)$ let us define the {\it $S$-resolvent set} and the {\it $S$-spectrum}
\begin{equation}\label{Eq_S_spectrum}
\rho_S(T):=\big\{s\in\mathbb{R}^{n+1}\;\big|\;Q_s[T]^{-1}\in\mathcal{B}(V)\big\}\qquad\text{and}\qquad\sigma_S(T):=\mathbb{R}^{n+1}\setminus\rho_{S}(T).
\end{equation}
Moreover, for every $s\in\rho_S(T)$ we define the {\it left} and the {\it right $S$-resolvent operator}
\begin{equation}\label{Eq_SL_SR}
S_L^{-1}(s,T):=Q_s[T]^{-1}\overline{s}-TQ_s[T]^{-1}\qquad\text{and}\qquad S_R^{-1}(s,T):=(\overline{s}-T)Q_s[T]^{-1}.
\end{equation}
\end{definition}

\section{The $S$-spectrum with boundary conditions}\label{SBOUND}

Let $V$ be a Banach module over the Clifford algebra $\R_n$. For a right-linear operator $T : \dom(T)\subseteq V \to V$ and $s \in \R^{n+1}$ we define
\[Q_{s}[T] := T^2-2 s_0 T+|s|^2 I: \dom(T^2)\subseteq V\to V.\]
We consider a right submodule $B \subseteq \dom (T)$ and define
\begin{equation*}
Q_{s,B}[T] := Q_s[T]_{|\dom_B(T^2)}:\dom_B(T^2) \subseteq V \to V
\end{equation*}
with
\[ \dom_B(T^2) :=  \dom(T^2) \cap B. \]
The following classifications on the domains associated with the $S$-resolvent operators is crucial for the
understanding of the theory on the $S$-spectrum:
\begin{enumerate}
\item[(i)]
$\dom(T)$ is the domain associated with the operator $T$,
\item[(ii)]
$\dom_B(T^2)$ is the domain associated with the $S$-spectrum.
\end{enumerate}
\begin{definition} 
We define the $S$-resolvent set with boundary conditions $B$ of $T$ as
\begin{equation*}
\rho_{S,B}(T)=\left\{s\in \R^{n+1}\ :\ Q_{s,B}[T]^{-1}\in \mathcal{B}(V)\ \right\}
\end{equation*}
and the $S$-spectrum $\sigma_{S,B}(T)$ with boundary conditions $B$  of $T$ as
$$
\sigma_{S,B}(T)=\R^{n+1}\setminus \rho_{S,B}(T).
$$
\end{definition}

\begin{definition}
Let $T$ be a closed right-linear operator and $B \subseteq \dom(T)$ be a right submodule of $V$. For $s\in\rho_{S,B}(T)$, we define the {\em left $S$-resolvent operator} as
\[
S_{L,B}^{-1}(s,T) := Q_{s,B}[T]^{-1}\overline{s}-TQ_{s,B}[T]^{-1},
\]
and the {\em right $S$-resolvent operator} as
$$
S_{R,B}^{-1}(s,T) = (\overline{s}-T)Q_{s,B}[T]^{-1}.
$$
\end{definition}

\begin{remark}
We remark the following facts:
\begin{itemize}
\item[(i)]
In the case we start with an operator $T: \dom(T)\subset V\to V$ such that $ \dom(T^2) \subseteq B \subseteq \dom (T)$, then clearly $\dom_B(T^2) = \dom(T^2)$ and $Q_s[T] = Q_{s,B}[T]$, hence the definition of $S$-resolvent and $S$-spectrum coincide with the classic ones.
\item[(ii)] Let $\Omega \subseteq \R^{n}$ be a bounded domain and consider the gradient operator with nonconstant coefficients $T = \sum_{i=1}^n e_i a_i(x) \pdv{}{x_i}$ with $\dom (T) = H^1(\Omega)$ as in \cite{CMS24}. If we set $B = H^1_0(\Omega)$, then $\dom_B(T^2) = H^2(\Omega)\cap H^1_0(\Omega)$, and finding $s \in \R^{n+1}$ is an equivalent problem to determine if for every $f \in L^2(\Omega)$ the boundary value problem
\begin{equation*}
    \begin{cases}
        T^2u-2s_0Tu + \abs{s}^2u = f\quad \text{on $\Omega$,}\\
        \tr (u) = 0 \quad \text{on $\partial\Omega$},
    \end{cases}
\end{equation*}
has a unique solution in $\dom(T^2)$ and
\begin{equation*}
    f \longmapsto u:=(T^2-2s_0T + \abs{s}^2 I)^{-1}f
\end{equation*}
is a bounded operator.
\end{itemize}

\end{remark}

A relevant difference with the classic theory of the $S$-spectrum is that, in general, $T$ and $Q_{s,B}[T]^{-1}$ do not commute on $\dom(T)$ for $s \in \rho_{S,B}(T)$. Indeed, since
\begin{equation}
\label{eq:commutativity T and Q_s}
    T Q_s[T] = Q_s[T] T\quad \text{on $\dom(T^3)$,}
\end{equation}
then, using the classic notion of $S$-resolvent, for $s \in \rho_{S}(T)$ we have
\begin{equation*}
    T Q_s[T]^{-1} = Q_s[T]^{-1} T \quad\text{on $\dom(T)$.}
\end{equation*}
However, if $s \in \rho_{S,B}(T)$ and $\dom(T^2) \not\subset  B$,  it is not true that $T$ and $Q_{s,B}[T]^{-1}$ commute on $\dom(T)$ and we will see in Section~\ref{section:analyticity} how this changes the study of the analyticity of the $S$-resolvent operators. In the next proposition, we analyze when this commutative property holds.

\begin{proposition}
\label{prop:commutativity Q_s,b and T}
    Let $T:\dom(T)\subseteq V \to V$ be a right-linear operator. Let $B \subseteq \dom(T)$ be a right submodule and $s \in \rho_{S,B}(T)$. Then
    \begin{equation*}
            TQ_{s,B}[T]^{-1}w = Q_{s,B}[T]^{-1}Tw\,\text{ if and only if }\,w=Q_{s,B}[T]v\text{ for }v \in \dom(T^3)\cap B\cap T^{-1}(B).
    \end{equation*}
\end{proposition}
\begin{proof}
    For every $v \in \dom(T^3) \cap B \cap T^{-1}(B)$ we know that $v \in \dom_B(T^2) $ as well as $Tv \in \dom_B(T^2)$. Hence, using Equation~\eqref{eq:commutativity T and Q_s}, we find that
    \begin{equation*}
        T Q_{s,B}[T]v = T Q_s[T]v=Q_s[T]Tv = Q_{s,B}[T]Tv.
    \end{equation*}
    Now let $s \in \rho_{S,B}(T)$, if $w = Q_{s,B}[T]v$ with $v \in \dom(T^3)\cap B\cap T^{-1}(B)$, then clearly $Q_{s,B}[T]^{-1}w = v \in \dom(T^3)\cap B\cap T^{-1}(B)$, and by the argument above, we have
    \begin{equation*}
        Q_{s,B}[T]TQ_{s,B}[T]^{-1}w = TQ_{s,B}[T]Q_{s,B}[T]^{-1}w=Tw,
    \end{equation*}
    and so
    \begin{equation*}
        TQ_{s,B}[T]^{-1}w=Q_{s,B}[T]^{-1}Tw.
    \end{equation*}
    Conversely, if $TQ_{s,B}[T]^{-1}w = Q_{s,B}[T]^{-1}Tw$ holds then
    \begin{equation*}
        T\bigl(Q_{s,B}[T]^{-1}w\bigr) = Q_{s,B}[T]^{-1}\bigl(Tw\bigr) \in \dom_B(T^2)= \dom(T^2) \cap B,
    \end{equation*}
    hence $Q_{s,B}[T]^{-1}w \in T^{-1}(B) \cap \dom(T^3)$. Moreover, by definition $Q_{s,B}[T]^{-1}w \in B$. Hence $w = Q_{s,B}[T] (Q_{s,B}[T]^{-1}w)$ with $ Q_{s,B}[T]^{-1}w \in  \dom(T^3)\cap B\cap T^{-1}(B)$. This concludes the proof.
\end{proof}

\begin{remark}
    Let $(T_{|B},B)$ the operator defined by $T_{|B} x=Tx$ for all $x \in B$, that is the restriction of the operator $T$ to the subspace of the boundary conditions $B$, we have
    \[\dom(T_{|B}^2)=\{u \in \dom(T_{|B})\,:\,Tu \in \dom(T_{|B})\}=\{u \in B\,:\,Tu\in B\}=B\cap T^{-1}(B),\]
    hence the previous proposition can be rewritten as: $TQ_{s,B}[T]^{-1}w=Q_{s,B}[T]^{-1}Tw$ if and only if $w \in Q_{s,B}[T](\dom(T^3)\cap\dom(T_{|B}^2))$.
\end{remark}

Since the commutativity of $Q_{s,B}[T]^{-1}$ and $T$ will play a crucial role throughout the paper, it may be useful to introduce the notion of commutator of these operators
\begin{align*}
    [T,Q_{s,B}[T]^{-1}]:=TQ_{s,B}[T]^{-1}-Q_{s,B}[T]^{-1}T:\dom(T)\subseteq V \to V.
\end{align*}
We can now rephrase the content of Proposition~\ref{prop:commutativity Q_s,b and T} as
\begin{equation*}
    \ker [T,Q_{s,B}[T]^{-1}] =\{ Q_{s,B}[T]v\,:\,v \in \dom(T^3)\cap B \cap T^{-1}(B)\}.
\end{equation*}

\section{Continuity of the $S$-resolvent operators with boundary conditions}\label{CONTIUN}

In \cite[Section 3.1]{FJBOOK}, the authors prove the continuous differentiability of the pseudo-resolvent operator $Q_s[T]^{-1}$ and the $S$-resolvent operators as preliminary results for the slice hyperholomorphicity of the $S$-resolvent operators. However, these proofs rely on the commutativity of the pseudo-resolvent $Q_s[T]^{-1}$ and the operator $T$ on the domain $\dom(T)$. In this section we revise the proofs of these facts in order to obtain similar properties for the pseudo-resolvent and the $S$-resolvent operators when they are associated with boundary conditions.

\begin{proposition}
\label{prop:serie per Q_s,b[T]}
    Let $T:\dom(T) \subseteq V \to V $ be a closed right-linear operator. Let $B \subseteq \dom(T)$ be a right submodule and $q \in \rho_{S,B}(T)$. For $s \in \R^{n+1}$, let
    \begin{equation}
        \label{eq:series boundary}
        \mathcal{J}(s) =\sum_{n=0}^\infty   Q_{q,B}[T]^{-1}(\Lambda(q,s)Q_{q,B}[T]^{-1})^n,
    \end{equation}
    where
\begin{equation*}
    \Lambda(q,s) \coloneqq Q_{q,B}[T]-Q_{s,B}[T] = 2(s_0-q_0)T+(\abs{q}^2-\abs{s}^2)I.
\end{equation*}
   Then the series converges to $Q_{s,B}[T]^{-1}$ uniformly in $\boundOP(V)$ on any of the closed axially-symmetric neighborhoods
    \begin{equation*}
        C_{\varepsilon}(q) = \{ s \in \R^{n+1}\,:\,d_S(s,q)\le \varepsilon\}
    \end{equation*}
    of $q$ with
    \begin{equation*}
        d_{S}(s,q) = \max\{2\abs{s_0-q_0},\abs{\abs{s}^2-\abs{q}^2}\}
    \end{equation*}
    and
    \begin{equation*}
        \varepsilon < \frac{1}{\norm{T Q_{q,B}[T]^{-1}}+\norm{Q_{q,B}[T]^{-1}}}.
    \end{equation*}
    In particular, $C_\varepsilon(q) \subseteq \rho_{S,B}(T)$ and $\rho_{S,B}(T)$ is open.
\end{proposition}

\begin{proof}
 Let $q \in \rho_{S,B}(T)$.
 The sets $C_\varepsilon(q)$ can be proved to be closed axially-symmetric neighborhoods of $q$ as in \cite[Theorem 3.1.2]{FJBOOK}.
  Now let
  $$\varepsilon<1/(\norm{T Q_{q,B}[T]^{-1}}+\norm{Q_{q,B}[T]^{-1}})$$ and $s \in C_{\varepsilon}(q)$, we have
\begin{equation*}
    Q_{s,B}[T] = Q_{q,B}[T] - \Lambda(q,s),
\end{equation*}
and so
\begin{equation*}
    Q_{s,B}[T] = (I-\Lambda(q,s)Q_{q,B}[T]^{-1})Q_{q,B}[T].
\end{equation*}
Now note that for any $s \in C_\varepsilon(q)$ we have
 \begin{align*}
     \norm{\Lambda(q,s)Q_{q,B}[T]^{-1}} &\le 2\abs{s_0-q_0}\norm{T Q_{q,B}[T]^{-1}}+\abs{\abs{q}^2-\abs{s}^2}\norm{Q_{q,B}[T]^{-1}}\\
     &\le d_S(s,q) (\norm{T Q_{q,B}[T]^{-1}}+\norm{Q_{q,B}[T]^{-1}})\\
     &\le \varepsilon  (\norm{T Q_{q,B}[T]^{-1}}+\norm{Q_{q,B}[T]^{-1}}) <1,
 \end{align*}
which implies that $(I-\Lambda(q,s)Q_{q,B}[T]^{-1})$ is boundedly invertible, and so $Q_{s,B}[T]$ is bijective with bounded inverse
\begin{equation*}
    Q_{s,B}[T]^{-1} = Q_{q,B}[T]^{-1}(I-\Lambda(q,s)Q_{q,B}[T]^{-1})^{-1}\in\boundOP(V).
\end{equation*}
This means that $s \in \rho_{S,B}(T)$ and so $\rho_{S,B}(T)$ is open. Moreover we obtain
\begin{equation*}
    Q_{s,B}[T]^{-1} = Q_{q,B}[T]^{-1}\sum_{n=0}^\infty (\Lambda(q,s)Q_{q,B}[T]^{-1})^n = \sum_{n=0}^\infty Q_{q,B}[T]^{-1}(\Lambda(q,s)Q_{q,B}[T]^{-1})^n.
\end{equation*}
\end{proof}

\begin{proposition}
\label{prop:Q_s and TQ_s are C^1}
    Let $T:\dom(T) \subseteq V \to V$ be a closed right-linear operator. Let $B \subseteq \dom(T)$ be a right submodule and suppose $\rho_{S,B}(T) \ne \emptyset$. Then the operator-valued functions
    \begin{equation*}
        s :\rho_{S,B}(T)\longmapsto Q_{s,B}[T]^{-1} \in \boundOP(V)
    \end{equation*}
    and
    \begin{equation*}
        s :\rho_{S,B}(T)\longmapsto TQ_{s,B}[T]^{-1} \in \boundOP(V)
    \end{equation*}
    are continuously differentiable.
\end{proposition}
\begin{proof}
   Let $q \in \rho_{S,B}(T)$, according to Proposition~\ref{prop:serie per Q_s,b[T]} it is possible to write in a neighborhood of $q$
\begin{align*}
    Q_{s,B}[T]^{-1} &=\sum_{n=0}^\infty Q_{q,B}[T]^{-1}(\Lambda(q,s)Q_{q,B}[T]^{-1})^n\\
    &= Q_{q,B}[T]^{-1} + \sum_{n=1}^\infty Q_{q,B}[T]^{-1}(\Lambda(q,s)Q_{q,B}[T]^{-1})^n
\end{align*}
Moreover the proof of Proposition~\ref{prop:serie per Q_s,b[T]} shows that
\begin{equation*}
    \norm{\Lambda(q,s)Q_{q,B}[T]^{-1}} \le d_S(s,q) (\norm{T Q_{q,B}[T]^{-1}}+\norm{Q_{q,B}[T]^{-1}})\longrightarrow 0 \text{ as $s \to q$,}
\end{equation*}
hence, since the series converges uniformly, we deduce that
$$Q_{s,B}[T]^{-1} \to Q_{q,B}[T]^{-1}\ \ {\rm in}\  \ \boundOP(V).$$
Now let $s \in C_{\varepsilon}(q)$ with
\begin{equation*}
    \varepsilon=\frac{1}{2\bigl(\norm{TQ_{q,B}[T]^{-1}}+\norm{Q_{q,B}[T]^{-1}}\bigr)}.
\end{equation*}
Since $T$ is closed, we can apply $T$ to the series expansion of $Q_{s,B}[T]$ and find
\begin{align*}
    T Q_{s,B}[T]^{-1} = TQ_{q,B}[T]^{-1} \sum_{n=0}^\infty(\Lambda(q,s)Q_{q,B}[T]^{-1})^n,
\end{align*}
moreover, as before, we have
\begin{align*}
  \norm{\Lambda(q,s)Q_{q,B}[T]^{-1}} &\le d(s,q)\bigl(\norm{TQ_{q,B}[T]^{-1}}+\norm{Q_{q,B}[T]^{-1}}\bigr)  \\
  &\le \varepsilon \bigl(\norm{TQ_{q,B}[T]^{-1}}+\norm{Q_{q,B}[T]^{-1}}\bigr) < \frac{1}{2},
\end{align*}
and so
\begin{align*}
    \norm{T Q_{s,B}[T]^{-1}} & \le \norm{TQ_{q,B}[T]^{-1}} \sum_{n=0}^\infty\norm{\Lambda(q,s)Q_{q,B}[T]^{-1}}^n \\
    &\le \norm{TQ_{q,B}[T]^{-1}} \sum_{n=0}^\infty \frac{1}{2^n} = 2 \norm{TQ_{q,B}[T]^{-1}}.
\end{align*}
Now observe that
\begin{align*}
    T Q_{s,B}[T]^{-1}- TQ_{q,B}[T]^{-1} &= T Q_{s,B}[T]^{-1}(I-Q_{s,B}[T]Q_{q,B}[T]^{-1}) \\
    &= TQ_{s,B}[T]^{-1} (Q_{q,B}[T]-Q_{s,B}[T])Q_{q,B}[T]^{-1}\\
    &= TQ_{s,B}[T]^{-1}\Lambda(q,s)Q_{q,B}[T]^{-1},
\end{align*}
hence we find
\begin{align*}
    \norm{T Q_{s,B}[T]^{-1}- TQ_{q,B}[T]^{-1}} &\le \norm{TQ_{s,B}[T]^{-1}}\norm{\Lambda(q,s)Q_{q,B}[T]^{-1}} \\
    &\le 2\norm{TQ_{q,B}[T]^{-1}}\norm{\Lambda(q,s)Q_{q,B}[T]^{-1}} \longrightarrow 0\, \text{ as $s \to q$,}
\end{align*}
since we showed before that $\norm{\Lambda(q,s)Q_{q,B}[T]^{-1}} \to 0$ as $s \to q$.\\
This proves that the two functions are continuous, it remains to prove that they are continuously differentiable.
For $s = x + J y\in\rho_{S,B}(T)$, we have
by a direct computation via the definition of derivative
\begin{align*}
\pdv{}{x} Q_{s,B}[T]^{-1}&=\lim_{\mathbb{R}\ni h\to 0}\frac{1}{h}\Big(Q_{x+Jy+h,B}[T]^{-1}-Q_{x+Jy,B}[T]^{-1}\Big)
\\&=\lim_{\mathbb{R}\ni h\to 0}\frac{1}{h}Q_{x+Jy,B}[T]^{-1}\Big(2hT-2hx-h^2\Big)Q_{x+Jy+h,B}[T]^{-1}
\end{align*}
and so we have
\begin{equation}
\label{eq:derivative in u of Q_s}
\frac{\partial}{\partial x} Q_{s,B}[T]^{-1}=2Q_{s,B}[T]^{-1}(T-x)Q_{s,B}[T]^{-1} = 2 Q_{s,B}[T]^{-1}(TQ_{s,B}[T])^{-1}-2xQ_{s,B}[T]^{-2},
\end{equation}
which is a sum of products of continuous functions in $s$ and so it is continuous in $s$.
Similarly we find
\begin{align*}
\pdv{}{y} Q_{s,B}[T]^{-1} &=\lim_{\mathbb{R}\ni h\to 0}\frac{1}{h}\Big(Q_{x+Jy+Jh,B}[T]^{-1}-Q_{x+Jy,B}[T]^{-1}\Big)\\
&=\lim_{\mathbb{R}\ni h\to 0}\frac{1}{h}Q_{x+Jy,B}[T]^{-1}\Big(Q_{x+Jy,B}[T]-Q_{x+Jy+Jh,B}[T]\Big)Q_{x+Jy+Jh,B}[T]^{-1}
\\&=\lim_{\mathbb{R}\ni h\to 0}\frac{1}{h}Q_{x+Jy,B}[T]^{-1}\Big(-2yh-h^2\Big)Q_{x+Jy+Jh,B}[T]^{-1},
\end{align*}
since
\begin{align*}
    Q_{x+Jy,B}[T]-Q_{x+Jy+Jh,B}[T]&=T^2-2 x T+x^2+y^2-(T^2-2 x T+x^2+(y+h)^2)\\
    &=y^2-(y+h)^2=-2yh-h^2,
\end{align*}
so we find
\begin{equation}
\label{eq:derivative in v of Q_s}
\frac{\partial}{\partial y} Q_{s,B}[T]^{-1}=-2yQ_{s,B}[T]^{-2},
\end{equation}
which, as before, is continuous in $s$.\\
Now, in order to prove that $Q_{s,B}[T]^{-1}$ is continuosly differentiable, we write the variable $s\in \R^{n+1}$ in terms of its $n+1$ real coordinates as $s = s_0 + \sum_{i=1}^n s_ie_i$ and we consider the function $s \mapsto Q_{s,B}[T]^{-1}$ as a composition of the following functions:
\begin{equation*}
    s \longmapsto \begin{pmatrix}
        x \\
        y
    \end{pmatrix} = \begin{pmatrix}
        s_0 \\
        \sqrt{s_1^2+\dots+s_n^2}
    \end{pmatrix} \longmapsto Q_{x+Jy,B}[T]^{-1},
\end{equation*}
where $J \in \mathbb{S}$ is fixed.
Then the partial derivative with respect to $s_0$ corresponds to the partial derivative with respect to $x$, which exists and is continuous as we proved before. Moreover the partial derivative with respect to $s_i$, with $1\le i \le n$ if $y \ne 0$ can be computed as
\begin{equation*}
    \pdv{}{s_i}Q_{s,B}[T]^{-1} = -2y Q_{s,B}[T]^{-2}\pdv{v}{s_i} = -2s_i Q_{s,B}[T]^{-1},
\end{equation*}
otherwise, if $y = 0$ (and so $s \in \R$), we can simply choose $J = e_i$ and the partial derivative with resepect to $s_i$ coincides with the partial derivative with respect to $y$. In both cases, the derivative with respect to $s_i$ exists and is continuous, and this proves that the function is continuosly differentiable.
Now we do the same for $TQ_{s,B}[T]^{-1}$, we see that
\begin{align*}
\pdv{}{x}T Q_{s,B}[T]^{-1}&=\lim_{\mathbb{R}\ni h\to 0}\frac{1}{h}\Big(TQ_{x+Jy+h,B}[T]^{-1}-TQ_{x+Jy,B}[T]^{-1}\Big)
\\&=\lim_{\mathbb{R}\ni h\to 0}\frac{1}{h}TQ_{x+Jy,B}[T]^{-1}\Big(2hT-2hx-h^2\Big)Q_{x+Jy+h,B}[T]^{-1}\\
&=\lim_{\mathbb{R}\ni h\to 0}TQ_{x+Jy,B}[T]^{-1}\Big(2T-2x-h\Big)Q_{x+Jy+h,B}[T]^{-1},
\end{align*}
which gives
\begin{equation}
    \label{eq:derivative in u TQ_s}
    \pdv{}{x}T Q_{s,B}[T]^{-1}= 2 TQ_{s,B}[T]^{-1}(TQ_{s,B}[T]^{-1})-2xTQ_{s,B}[T]^{-2},
\end{equation}
and
\begin{align*}
\pdv{}{y} TQ_{s,B}[T]^{-1} &=\lim_{\mathbb{R}\ni h\to 0}\frac{1}{h}\Big(TQ_{x+Jy+Jh,B}[T]^{-1}-TQ_{x+Jy,B}[T]^{-1}\Big)\\
&=\lim_{\mathbb{R}\ni h\to 0}\frac{1}{h}TQ_{x+Jy,B}[T]^{-1}\Big(Q_{x+Jy,B}[T]-Q_{x+Jy+Jh,B}[T]\Big)Q_{x+Jy+Jh,B}[T]^{-1}
\\&=\lim_{\mathbb{R}\ni h\to 0}\frac{1}{h}TQ_{x+Jy,B}[T]^{-1}\Big(-2yh-h^2\Big)Q_{x+Jy+Jh,B}[T]^{-1},
\end{align*}
which gives
\begin{equation}
\label{eq:derivative in v TQ_s}
     \pdv{}{y} TQ_{s,B}[T]^{-1} = -2y TQ_{s,B}[T]^{-2}.
\end{equation}
The derivatives of $TQ_{s,B}[T]^{-1}$ are continuous functions in $s$ and, with the same argument as before, this proves that $TQ_{s,B}[T]^{-1}$ is continuously differentiable.
\end{proof}

The next result now is a simple consequence of Proposition~\ref{prop:Q_s and TQ_s are C^1}, since the $S$-resolvent operators are sums of products of continuously differentiable functions.

\begin{corollary}
\label{cor:S-resolvent operators are C^1}
     Let $T:\dom(T) \subseteq V \to V$ be a closed right-linear operator. Let $B \subseteq \dom(T)$ be a right submodule and suppose $\rho_{S,B}(T) \ne \emptyset$. Then the left $S$-resolvent operator
    \begin{equation*}
        s:\rho_{S,B}(T) \longmapsto S^{-1}_{L,B}(s,T) \in \boundOP(V)
    \end{equation*}
    and the right $S$-resolvent operator
    \begin{equation*}
        s:\rho_{S,B}(T) \longmapsto S^{-1}_{R,B}(s,T) \in \boundOP(V)
    \end{equation*}
    are continuously differentiable functions in $s$.
\end{corollary}

\section{Analyticity of the $S$-resolvent operators with boundary conditions}
\label{section:analyticity}

In this section we investigate the analyticity of the $S$-resolvent operators when the $S$-spectrum is associated with boundary conditions. While the continuous differentiability of the $S$-resolvent operators can be extended to this case with small modifications, the slice hyperholomorphicity of the $S$-resolvent operators strongly depends on the possibility of commuting $T$ and $Q_s[T]^{-1}$ on the range of $\dom(T)$ and so must be explored more carefully.\\
Firstly, we show that $S^{-1}_{L,B}(s,T)$ and $S^{-1}_{R,B}(s,T)$ are, respectively, a right and a left slice function on $\rho_{S,B}(T)$.

\begin{proposition}
    \label{prop:S resolvent are slice}
    Let $T:\dom(T) \subseteq V \to V$ be a closed right-linear operator. Let $B \subseteq \dom(T)$ be a right submodule and suppose $\rho_{S,B}(T) \ne \emptyset$. Then for every $s = x+Jy \in \rho_{S,B}(T)$, with $J \in \mathbb{S}$, we have
    \begin{equation*}
        S^{-1}_{L,B}(s,T) = f_0(x,y) + f_1(x,y) J\quad\text{and}\quad S^{-1}_{R,B}(s,T) = f_0(x,y)+Jf_1(x,y),
    \end{equation*}
    with
\begin{align*}
    f_0(x,y) &= -TQ_{s,B}[T]^{-1}+Q_{s,B}[T]^{-1}x,\\
    f_1(x,y) &= -Q_{s,B}[T]^{-1}y.
\end{align*}
that is $S^{-1}_{L,B}(s,T)$, resp. $S^{-1}_{R,B}(s,T)$, is a continuosly differentiable right, resp. left, slice function on $\rho_{S,B}(T)$. Moreover
\begin{align}
   \notag \frac{\partial f_0}{\partial x}(x,y) =& -2T Q_{s,B}[T]^{-1}T Q_{s,B}[T]^{-1}+2xT Q_{s,B}[T]^{-2}\\
   \label{eq: partial u f_0}
    &+2x Q_{s,B}[T]^{-1} T  Q_{s,B}[T]^{-1}-2x^2 Q_{s,B}[T]^{-2}+ Q_{s,B}[T]^{-1},\\
    \label{eq:partial v f_0}
    \frac{\partial f_0}{\partial y}(x,y) =&\,\,2y T  Q_{s,B}[T]^{-2}-2xy  Q_{s,B}[T]^{-2},\\
    \label{eq:partial u f_1}
    \frac{\partial f_1}{\partial x}(x,y) =& -2y  Q_{s,B}[T]^{-1} T  Q_{s,B}[T]^{-1} + 2x y  Q_{s,B}[T]^{-2},\\
    \label{eq:partial v f_1}
    \frac{\partial f_1}{\partial y}(x,y) =& 2y^2  Q_{s,B}[T]^{-2}- Q_{s,B}[T]^{-1}.
\end{align}
\end{proposition}
\begin{proof}
    We already know from Corollary~\ref{cor:S-resolvent operators are C^1} that $S^{-1}_{L,B}(s,T)$ and $S^{-1}_{R,B}(s,T)$ are continuosly differentiable functions. Now let $s = x+Jy \in \rho_{S,B}(T)$, we have
    \begin{align*}
        S^{-1}_{L,B}(s,T) &= Q_{s,B}[T]^{-1}\ov{s}-TQ_{s,B}[T]^{-1}\\
        &= Q_{s,B}[T]^{-1}(x-Jy)-TQ_{s,B}[T]^{-1}\\
        &= -TQ_{s,B}[T]^{-1}+Q_{s,B}[T]^{-1}x-yQ_{s,B}[T]^{-1}J\\
        &= f_0(x,y)+f_1(x,y)J,
    \end{align*}
    and analogously
        \begin{align*}
        S^{-1}_{R,B}(s,T) &=(\ov{s}-T) Q_{s,B}[T]^{-1}\\
        &= (x-Jy-T)Q_{s,B}[T]^{-1}\\
        &= -TQ_{s,B}[T]^{-1}+Q_{s,B}[T]^{-1}x+J(-yQ_{s,B}[T]^{-1})\\
        &= f_0(x,y)+Jf_1(x,y),
    \end{align*}
    which proves that $S^{-1}_{L,B}(s,T)$, resp. $S^{-1}_{R,B}(s,T)$, is a right, resp. left, slice function, since
    \begin{equation*}
        f_0(x,y) =  -TQ_{s,B}[T]^{-1}+Q_{s,B}[T]^{-1}x = f_0(x,-y),
    \end{equation*}
    and
    \begin{equation*}
        f_1(x,-y) = Q_{s,B}[T]^{-1}y = -f_1(x,y).
    \end{equation*}
    Using \eqref{eq:derivative in u of Q_s}, \eqref{eq:derivative in v of Q_s}, \eqref{eq:derivative in u TQ_s} and \eqref{eq:derivative in v TQ_s} we find
    \begin{align*}
        \frac{\partial f_0}{\partial x}(x,y) &= - \pdv{}{x} \bigl(T Q_{s,B}[T]^{-1}\bigr) + \pdv{}{x}\bigl( Q_{s,B}[T]^{-1}\bigr)x + Q_{s,B}[T]^{-1}
        \\
        &= -2 TQ_{s,B}[T]^{-1}TQ_{s,B}[T]^{-1}+2xTQ_{s,B}[T]^{-2}+ 2x Q_{s,B}[T]^{-1}TQ_{s,B}[T]^{-1}
        \\
        &-2x^2Q_{s,B}[T]^{-2}+ Q_{s,B}[T]^{-1},
    \end{align*}
    and
    \begin{align*}
        \frac{\partial f_0}{\partial y}(x,y) &=  - \pdv{}{y} \bigl(T Q_{s,B}[T]^{-1}\bigr) + \pdv{}{y}\bigl( Q_{s,B}[T]^{-1}\bigr)x \\
        &= 2y TQ_{s,B}[T]^{-2}-2xyQ_{s,B}[T]^{-2}.
    \end{align*}
    Similarly
    \begin{align*}
        \frac{\partial f_1}{\partial x}(x,y) &= - \pdv{}{x} \bigl( Q_{s,B}[T]^{-1}  \bigr)y\\
        &= -2y Q_{s,B}[T]^{-1}TQ_{s,B}[T]^{-1}+2xyQ_{s,B}[T]^{-2},
    \end{align*}
    and
    \begin{align*}
       \frac{\partial f_1}{\partial y}(x,y) &= - \pdv{}{y} \bigl( Q_{s,B}[T]^{-1}  \bigr)y-Q_{s,B}[T]^{-1}\\
        &= 2y^2Q_{s,B}[T]^{-2}-Q_{s,B}[T]^{-1}.
    \end{align*}
\end{proof}

Now we come to the main topic concerning the $S$-resolvent operators with boundary conditions, that is its analyticity. This property is crucial since the $S$-functional calculus, which generalises the Riesz-Dunford functional calculus to the noncommutative setting, is based on the Cauchy formula for slice hyperholomorphic functions and requires the holomorphicity of the resolvent operators. Moreover, for the same reason, also the $H^\infty$-functional calculus for Clifford operators requires the same property. \\
Since, according to Proposition~\ref{prop:S resolvent are slice}, the $S$-resolvent operators are continuously differentiable slice functions, now we must investigate if the Cauchy-Riemann equations
\begin{align}
\label{eq:first cauchy riemann equation}
    \frac{\partial f_0}{\partial x}(x,y) - \frac{\partial f_1}{\partial y}(x,y) &= 0,\\
\label{eq:second cauchy riemann equation}
    \frac{\partial f_0}{\partial y}(x,y) + \frac{\partial f_1}{\partial x}(x,y) &= 0
\end{align}
are satisfied on $V$ or, at least, on some subspace.
From \eqref{eq: partial u f_0} and \eqref{eq:partial v f_1}, Equation \eqref{eq:first cauchy riemann equation} becomes
\begin{equation}
\label{eq:first CR for S resolvent op's}
    -TQ_{s,B}[T]^{-1}TQ_{s,B}[T]^{-1}+xTQ_{s,B}[T]^{-2}+ x Q_{s,B}[T]^{-1}TQ_{s,B}[T]^{-1}+ Q_{s,B}[T]^{-1}-\abs{s}^2Q_{s,B}[T]^{-2}=0,
\end{equation}
moreover using \eqref{eq:partial v f_0} and \eqref{eq:partial u f_1}, Equation \eqref{eq:second cauchy riemann equation} becomes
\begin{equation}
    \label{eq:second CR for S resolvent op's}
    y(T Q_{s,B}[T]^{-2}- Q_{s,B}[T]^{-1} T Q_{s,B}[T]^{-1})=0.
\end{equation}

The following proposition provides a sufficient condition to identify points in $V$ where the Cauchy-Riemann equations for the $S$-resolvent operators are satisfied.

\begin{proposition}
\label{prop:sufficient condition for cauchy riemann}
    Let $T:\dom(T) \subseteq V \to V$ be a closed right-linear operator. Let $B \subseteq \dom(T)$ be a right submodule. Let $s \in \rho_{S,B}(T)$ and $v \in V$, suppose that one of the following statements holds:
    \begin{enumerate}
        \item  $s \notin \R$ and $Q_{s,B}[T]^{-1}v \in \ker [T,Q_{s,B}[T]^{-1}]$;
        \item $s \in \R$ and $[T,Q_{s,B}[T]^{-1}]Q_{s,B}[T]^{-1}v \in \ker (T-sI)$,
    \end{enumerate}
    then conditions \eqref{eq:first CR for S resolvent op's} and \eqref{eq:second CR for S resolvent op's} hold in $v$.
\end{proposition}
\begin{proof}
    Suppose that $s \notin \R$ and $Q_{s,B}[T]^{-1}v \in \ker [T,Q_{s,B}[T]^{-1}]$, that is
    \begin{equation*}
       0= (T Q_{s,B}[T]^{-1}-Q_{s,B}[T]^{-1}T)Q_{s,B}[T]^{-1}v = T Q_{s,B}[T]^{-2}v-Q_{s,B}[T]^{-1}T Q_{s,B}[T]^{-1}v,
    \end{equation*}
then clearly \eqref{eq:second CR for S resolvent op's} holds in $v$. Moreover
\begin{align*}
    &(-TQ_{s,B}[T]^{-1}TQ_{s,B}[T]^{-1}+xTQ_{s,B}[T]^{-2}+ x Q_{s,B}[T]^{-1}TQ_{s,B}[T]^{-1}+ Q_{s,B}[T]^{-1}-\abs{s}^2Q_{s,B}[T]^{-2})v\\
    &= (-TQ_{s,B}[T]^{-1}T+xTQ_{s,B}[T]^{-1}+xQ_{s,B}[T]^{-1}T-\abs{s}^2Q_{s,B}[T]^{-1})Q_{s,B}[T]^{-1}v+Q_{s,B}[T]^{-1}v\\
    &= -(T^2 Q_{s,B}[T]^{-1}-2xTQ_{s,B}[T]^{-1}+\abs{s}^2Q_{s,B}[T]^{-1})Q_{s,B}[T]^{-1}v+Q_{s,B}[T]^{-1}v\\
    &=-Q_s[T] Q_{s,B}[T]^{-1}Q_{s,B}[T]^{-1}v+Q_{s,B}[T]^{-1}v\\
    &= -Q_{s,B}[T]Q_{s,B}[T]^{-1}Q_{s,B}[T]^{-1}v+Q_{s,B}[T]^{-1}v\\
    &= -Q_{s,B}[T]^{-1}v+Q_{s,B}[T]^{-1}v = 0,
\end{align*}
which proves that \eqref{eq:first CR for S resolvent op's} holds in $v$.\\
Now suppose that $s=x \in \R$, then $y=0$ and so Equation \eqref{eq:second CR for S resolvent op's} becomes trivial. Moreover, by assumption, we have
\begin{align*}
    0 &= T [T,Q_{x,B}[T]^{-1}]Q_{x,B}[T]^{-1}v-x[T,Q_{x,B}[T]^{-1}]Q_{x,B}[T]^{-1}v\\
    &=T^2Q_{x,B}[T]^{-2}v-TQ_{x,B}[T]^{-1}TQ_{x,B}[T]^{-1}v-xTQ_{x,B}[T]^{-2}v+xQ_{x,B}[T]^{-1}TQ_{x,B}[T]^{-1}v,
\end{align*}
which gives
\begin{equation}
    \label{eq:reformulation hyp 2 prop sufficient condition}
    -TQ_{x,B}[T]^{-1}TQ_{x,B}[T]^{-1}v+xQ_{x,B}[T]^{-1}TQ_{x,B}[T]^{-1}v = -T^2Q_{x,B}[T]^{-2}v+xTQ_{x,B}[T]^{-2}v.
\end{equation}
Hence, substituting \eqref{eq:reformulation hyp 2 prop sufficient condition} in \eqref{eq:first CR for S resolvent op's}, we find that
\begin{align*}
    &-TQ_{x,B}[T]^{-1}TQ_{x,B}[T]^{-1}v+xTQ_{x,B}[T]^{-2}v+ x Q_{x,B}[T]^{-1}TQ_{x,B}[T]^{-1}v
    \\
    &+ Q_{x,B}[T]^{-1}v-x^2Q_{x,B}[T]^{-2}v \\
    &=-T^2Q_{x,B}[T]^{-2}v+2xTQ_{x,B}[T]^{-2}v-x^2Q_{x,B}[T]^{-2}v+Q_{x,B}[T]^{-1}v\\
    &=-Q_{x}[T]Q_{x,B}[T]^{-2}v+Q_{x,B}[T]^{-1}v = -Q_{x,B}[T]Q_{x,B}[T]^{-2}v+Q_{x,B}[T]^{-1}v\\
    &=-Q_{x,B}[T]^{-1}v+Q_{x,B}[T]^{-1}v = 0.
\end{align*}
that is, the first Cauchy-Riemann equation holds in $v$.
\end{proof}

In the next proposition we show that these conditions are actually also necessary.

\begin{proposition}
\label{prop:necessary condition for cauchy riemann}
    Let $T:\dom(T) \subseteq V \to V$ be a closed right-linear operator. Let $B \subseteq \dom(T)$ be a right submodule. Let $s \in \rho_{S,B}(T)$, suppose that conditions \eqref{eq:first CR for S resolvent op's} and \eqref{eq:second CR for S resolvent op's} hold in $v \in V$.
    \begin{enumerate}
        \item  If $s \notin \R$, then $Q_{s,B}[T]^{-1}v \in \ker [T,Q_{s,B}[T]^{-1}]$.
        \item If $s \in \R$, then $[T,Q_{s,B}[T]^{-1}]Q_{s,B}[T]^{-1}v \in \ker (T-sI)$.
    \end{enumerate}
\end{proposition}
\begin{proof}
    Let $s = x+Jy \in \rho_{S,B}(T)$, $J \in \mathbb{S}$, and let $u = Q_{s,B}[T]^{-1}v$. First suppose that $s \notin \R$, that is $y \ne 0$, then Equation \eqref{eq:second CR for S resolvent op's} becomes
    \begin{equation*}
       0= T Q_{s,B}[T]^{-1}Q_{s,B}[T]^{-1}v-Q_{s,B}[T]^{-1}T Q_{s,B}[T]^{-1}v =TQ_{s,B}[T]^{-1}u-Q_{s,B}[T]^{-1}Tu,
    \end{equation*}
    that is $u \in \ker[T, Q_{s,B}[T]^{-1}]$. Now suppose that $s \in \R$, that is $s=x$ and $y=0$, Equation \eqref{eq:second CR for S resolvent op's} becomes trivial, so we must work on Equation \eqref{eq:first CR for S resolvent op's}. We have
    \begin{align*}
        0 =& -T Q_{x,B}[T]^{-1}T Q_{x,B}[T]^{-1}v+xT Q_{x,B}[T]^{-2}v+x Q_{x,B}[T]^{-1} T  Q_{x,B}[T]^{-1}v\\
        &+ Q_{x,B}[T]^{-1}v-x^2 Q_{x,B}[T]^{-2}v\\
        =& -T Q_{x,B}[T]^{-1}Tu+xT Q_{x,B}[T]^{-1}u+x Q_{x,B}[T]^{-1} T  u+u-x^2 Q_{x,B}[T]^{-1}u
    \end{align*}
    and so
    \begin{align*}
        u &= T Q_{x,B}[T]^{-1}Tu-x(T Q_{x,B}[T]^{-1}u+ Q_{x,B}[T]^{-1} T  u)+x^2 Q_{x,B}[T]^{-1}u \\
        &= T(Q_{x,B}[T]^{-1}Tu-TQ_{x,B}[T]^{-1}u+TQ_{x,B}[T]^{-1}u)\\
        &-x(T Q_{x,B}[T]^{-1}u+ Q_{x,B}[T]^{-1} T  u+T Q_{x,B}[T]^{-1}u-TQ_{x,B}[T]^{-1}u) + x^2Q_{x,B}[T]^{-1}u\\
        &=T^2Q_{x,B}[T]^{-1}u-T([T,Q_{x,B}[T]^{-1}]u)-2xTQ_{x,B}[T]^{-1}u+x[T,Q_{x,B}[T]^{-1}]u+x^2Q_{x,B}[T]^{-1}u\\
        &= (T^2-2xT+x^2)Q_{x,B}[T]^{-1}u-(T-x)[T,Q_{x,B}[T]^{-1}]u\\
        &=Q_{x}[T]Q_{x,B}[T]^{-1}u-(T-x)[T,Q_{x,B}[T]^{-1}]u\\
        &=Q_{x,B}[T]Q_{x,B}[T]^{-1}u-(T-x)[T,Q_{x,B}[T]^{-1}]u\\
        &=u-(T-x)[T,Q_{x,B}[T]^{-1}]u,
    \end{align*}
    that is
    \begin{equation*}
        (T-x)[T,Q_{x,B}[T]^{-1}]u=0,
    \end{equation*}
    which means that $[T,Q_{x,B}[T]^{-1}]u \in \ker (T-x)$.
\end{proof}

\begin{remark}
    Note that $w=[T,Q_{s,B}[T]^{-1}]Q_{s,B}[T]^{-1}v \ne 0$ means that 
    \[ Q_{s,B}[T]^{-1}v \notin  \{ Q_{s,B}[T]u\,:\,u \in \dom(T^3)\cap B \cap T^{-1}(B)\}\]
     which is possible only if $T Q_{s,B}[T]^{-2}v \notin B$. 
     In this case, 
     \[w =(TQ_{s,B}[T]^{-1}-Q_{s,B}[T]^{-1}T)Q_{s,B}[T]^{-1}v\]
     cannot belong to $B$, hence $w \in \ker (T-s)$ does not contradict that $Q_{s,B}[T] = (T-s)^2$ is invertible on $\dom_B(T^2)$ when $s \in \rho_{S,B}(T) \cap \R$.
\end{remark}

Putting all together we can state the following summary theorem.

\begin{theorem}
   Let $T:\dom(T) \subseteq V \to V$ be a closed right-linear operator and $B \subseteq \dom(T)$ be a right submodule. Let $s \in \rho_{S,B}(T)$ and $v \in V$. Then:
   \begin{enumerate}
       \item if $s \notin \R$, the Cauchy-Riemann equations \eqref{eq:first CR for S resolvent op's} and \eqref{eq:second CR for S resolvent op's} hold in $v$ if and only if \[Q_{s,B}[T]^{-1}v\in \ker [T,Q_{s,B}[T]^{-1}];\]
       \item if $s \in \R$, the Cauchy-Riemann equations \eqref{eq:first CR for S resolvent op's} and \eqref{eq:second CR for S resolvent op's} hold in $v$ if and only if
       \[[T,Q_{s,B}[T]^{-1}]Q_{s,B}[T]^{-1}v \in \ker (T-sI).\]
   \end{enumerate}
\end{theorem}

\begin{corollary}
     Let $T:\dom(T) \subseteq V \to V$ be a closed right-linear operator. If $ \dom(T^2) \subseteq B \subseteq \dom(T)$, then conditions \eqref{eq:first CR for S resolvent op's} and \eqref{eq:second CR for S resolvent op's} hold for all $v \in V$. In particular, $S^{-1}_{L,B}(s,T)$ (resp. $S^{-1}_{R,B}(s,T)$) is a slice right (resp. left) hyperholomorphic function.
\end{corollary}
\begin{proof}
    Indeed, in this case $\ker[T,Q_{s,B}[T]^{-1}]=\dom(T)$ and so for all $v \in V$ we have $Q_{s,B}[T]^{-1}v \in \dom(T^2)\subset \dom(T)$. Then use Proposition \ref{prop:sufficient condition for cauchy riemann}.
\end{proof}

In the next theorem we generalise the left and right $S$-resolvent equations to the case of $S$-spectrum with boundary conditions. While the left $S$-resolvent equation remains the same as in the classic case, the right $S$-resolvent equation is different and we must take into account also the commutator of $T$ and $Q_{s,B}[T]^{-1}$.

\begin{theorem}[The left and the right $S$-resolvent equations]
\label{thm:left and right S resolvent equations}
 Let $T:\dom(T) \subseteq V \to V$ be a closed right-linear operator. Let $B \subseteq \dom(T)$ be a right submodule. Let $s\in\rho_{S,B}(T)$, then the left $S$-resolvent operator satisfies the left $S$-resolvent equation
\begin{equation}\label{LeftSREQ}
S_{L,B}^{-1}(s,T)s - TS_{L,B}^{-1}(s,T) = I,
\end{equation}
and the right $S$-resolvent operator satisfies the right $S$-resolvent equation
\begin{equation}\label{RightSREQ}
sS_{R,B}^{-1}(s,T)v - S_{R,B}^{-1}(s,T)Tv = v+(\overline{s}-T)[T,Q_{s,B}[T]^{-1}]v,\quad\text{for all $v \in \dom(T)$}.
\end{equation}
\end{theorem}
\begin{proof}
We prove \eqref{LeftSREQ}. We have

\begin{align*}
 S_{L,B}^{-1}(s,T)s - TS_{L,B}^{-1}(s,T) &=(Q_{s,B}[T]^{-1}\overline{s}-TQ_{s,B}[T]^{-1})s -
T(Q_{s,B}[T]^{-1}\overline{s}-TQ_{s,B}[T]^{-1})\\
&=Q_{s,B}[T]^{-1}\overline{s}s-TQ_{s,B}[T]^{-1}s -
TQ_{s,B}[T]^{-1}\overline{s}+T^2Q_{s,B}[T]^{-1}
\\
&
=Q_{s,B}[T]^{-1}|s|^2-2s_0TQ_{s,B}[T]^{-1}s+T^2Q_{s,B}[T]^{-1}
\\
&
=(T^2-2s_0T+|s|^2)Q_{s,B}[T]^{-1}=I.
\end{align*}

Similarly, for $v \in \dom(T)$ we have
\begin{align*}
    sS_{R,B}^{-1}(s,T)v &- S_{R,B}^{-1}(s,T)Tv = s(\ov{s}-T)Q_{s,B}[T]^{-1}v-(\ov{s}-T)Q_{s,B}[T]^{-1}Tv\\
    &=\abs{s}^2 Q_{s,B}[T]^{-1}v-sTQ_{s,B}[T]^{-1}v-\ov{s}Q_{s,B}[T]^{-1}Tv+TQ_{s,B}[T]^{-1}Tv\\
    &=\abs{s}^2 Q_{s,B}[T]^{-1}v-sTQ_{s,B}[T]^{-1}v-\ov{s}TQ_{s,B}[T]^{-1}v+\ov{s}TQ_{s,B}[T]^{-1}v\\
    &-\ov{s}Q_{s,B}[T]^{-1}Tv+TQ_{s,B}[T]^{-1}Tv+T^2Q_{s,B}[T]^{-1}v-T^2Q_{s,B}[T]^{-1}v\\
    &= (T^2-2s_0T+\abs{s}^2)Q_{s,B}[T]^{-1}v+\ov{s}[T,Q_{s,B}[T]^{-1}]v-T[T,Q_{s,B}[T]^{-1}]v\\
    &=Q_{s,B}[T]Q_{s,B}[T]^{-1}v+(\ov{s}-T)[T,Q_{s,B}[T]^{-1}]v = v+(\ov{s}-T)[T,Q_{s,B}[T]^{-1}]v.
\end{align*}

\end{proof}

Now we are ready to show how the $S$-resolvent equation changes with boundary conditions. Since both left and right $S$-resolvent equations are used in the proof of the formula, and since Equation \eqref{RightSREQ} is different from the classic case, the new $S$-resolvent equation will be different and the commutator of $T$ and $Q_{s,B}[T]^{-1}$ will appear in this formula.

\begin{theorem}[$S$-resolvent equation]
 Let $T:\dom(T) \subseteq V \to V$ be a closed right-linear operator. Let $B \subseteq \dom(T)$ be a right submodule. If  $s,q \in  \rho_{S,B}(T)$ with $s\notin[q]$, then
\begin{multline}\label{resEQ}
S_{R,B}^{-1}(s,T)S_{L,B}^{-1}(q,T)=\bigl\{\big(S_{R,B}^{-1}(s,T)-S_{L,B}^{-1}(q,T)\big)q
-\overline{s}\bigl(S_{R,B}^{-1}(s,T)-S_{L,B}^{-1}(q,T)\bigr)\\
+\ov{s}(\ov{s}-T)[T,Q_{s,B}[T]^{-1}]S^{-1}_{L,B}(q,T)-(\ov{s}-T)[T,Q_{s,B}[T]^{-1}]S^{-1}_{L,B}(q,T)q\big\}(q^2-2s_0q+|s|^2)^{-1}.
\end{multline}
\end{theorem}
\begin{proof}
As in the classic case, the $S$-resolvent equation is deduced from the left and the right $S$-resolvent equation.
We are going to show that, for every $v\in V$, one has
\begin{align}
\notag S_{R,B}^{-1}&(s,T)S_{L,B}^{-1}(q,T)(q^2-2s_0q+|s|^2)v=\bigl\{\big(S_{R,B}^{-1}(s,T)-S_{L,B}^{-1}(q,T)\big)q
-\overline{s}\bigl(S_{R,B}^{-1}(s,T)-S_{L,B}^{-1}(q,T)\bigr)\\ \label{RESeqStart}
&+\ov{s}(\ov{s}-T)[T,Q_{s,B}[T]^{-1}]S^{-1}_{L,B}(q,T)x-(\ov{s}-T)[T,Q_{s,B}[T]^{-1}]S^{-1}_{L,B}(q,T)q\big\}v.
\end{align}
We then obtain  \eqref{resEQ} by replacing $v$ by $(q^2-2s_0q+|s|^2)^{-1}v$. For $w\in V$, the left $S$-resolvent equation \eqref{LeftSREQ} implies
\[ S_{R,B}^{-1}(s,T)S_{L,B}^{-1}(q,T)qw = S_{R,B}^{-1}(s,T)TS_{L,B}^{-1}(q,T)w + S_{R,B}^{-1}(s,T)w. \]
Since $ Q_{s,B}[T]^{-1}$ maps $V$ onto $\dom_B(T^2)$, then the left $S$-resolvent operator $$S_{L,B}^{-1}(s,T) =  Q_{s,B}[T]^{-1}\overline{s} - T Q_{s,B}[T]^{-1}$$ maps $V$ to $\dom(T)$ and so $S_{L,B}^{-1}(q,T)w\in\dom(T)$. The right $S$-resolvent equation \eqref{RightSREQ} yields
\begin{align}
\notag
&S_{R,B}^{-1}(s,T)S_{L,B}^{-1}(q,T)qw \\\label{SRSLsplit}
= &sS_{R,B}^{-1}(s,T)S_{L,B}^{-1}(q,T)w - S_{L,B}^{-1}(q,T)w -(\ov{s}-T)[T,Q_{s,B}[T]^{-1}]S^{-1}_{L,B}(q,T)w+ S_{R,B}^{-1}(s,T)w.
\end{align}
If we apply this identity with $w = qv$ we get
\begin{align*}
&S_{R,B}^{-1}(s,T)S_{L,B}^{-1}(q,T)(q^2-2s_0q+|s|^2)v
\\
&=S_{R,B}^{-1}(s,T)S_{L,B}^{-1}(q,T)q^2v -2s_0 S_{R,B}^{-1}(s,T)S_{L,B}^{-1}(q,T)qv+|s|^2 S_{R,B}^{-1}(s,T)S_{L,B}^{-1}(q,T) v
\\
&=sS_{R,B}^{-1}(s,T)S_{L,B}^{-1}(q,T)qv - S_{L,B}^{-1}(q,T)qv -(\ov{s}-T)[T,Q_{s,B}[T]^{-1}]S^{-1}_{L,B}(q,T)qv+ S_{R,B}^{-1}(s,T)qv
\\
 &-2s_0 S_{R,B}^{-1}(s,T)S_{L,B}^{-1}(q,T)qv + |s|^2 S_{R,B}^{-1}(s,T)S_{L,B}^{-1}(q,T) v.
\end{align*}
Applying identity \eqref{SRSLsplit} again with $w = v$ gives

\begin{align*}
&S_{R,B}^{-1}(s,T)S_{L,B}^{-1}(q,T)(q^2-2s_0q+|s|^2)v
\\
=&s^2S_{R,B}^{-1}(s,T)S_{L,B}^{-1}(q,T)v - sS_{L,B}^{-1}(q,T)v -s(\ov{s}-T)[T,Q_{s,B}[T]^{-1}]S^{-1}_{L,B}(q,T)v + sS_{R,B}^{-1}(s,T) v
\\
&- S_{L,B}^{-1}(q,T)qv -(\ov{s}-T)[T,Q_{s,B}[T]^{-1}]S^{-1}_{L,B}(q,T)qv+ S_{R,B}^{-1}(s,T)qv
\\
&
-2s_0 sS_{R,B}^{-1}(s,T)S_{L,B}^{-1}(q,T)v +2s_0 S_{L,B}^{-1}(q,T)v+2s_0(\ov{s}-T)[T,Q_{s,B}[T]^{-1}]S^{-1}_{L,B}(q,T)v
\\
& -2s_0S_{R,B}^{-1}(s,T)v+ |s|^2 S_{R,B}^{-1}(s,T)S_{L,B}^{-1}(q,T) v
\\
=&(s^2-2s_0s + |s|^2)S_{R,B}^{-1}(s,T)S_{L,B}^{-1}(q,T)v- (2s_0 -s)[S_{R,B}^{-1}(s,T)-S_{L,B}^{-1}(q,T)]
\\
&+[S_{R,B}^{-1}(s,T)- S_{L,B}^{-1}(q,T)]qv
\\
&+(2s_0-s)(\ov{s}-T)[T,Q_{s,B}[T]^{-1}]S^{-1}_{L,B}(q,T)v
-(\ov{s}-T)[T,Q_{s,B}[T]^{-1}]S^{-1}_{L,B}(q,T)qv.
\end{align*}
The identity $2s_0 = s + \overline{s}$ implies $s^2 - 2s_0s+ |s|^2 = 0$ and $2s_0 - s = \overline{s}$ and hence we obtain the desired equation \eqref{RESeqStart}.

\end{proof}

\section{Concluding remarks}
In contrast to complex analysis, the noncommutative setting of Clifford analysis admits multiple definitions of hyperholomorphicity. Consequently, the associated Cauchy formulas serve as the foundation for distinct spectral theories, each characterized by specific Cauchy kernels. Notably, the spectral theory based on the $S$-spectrum is intrinsically linked to slice hyperholomorphicity and draws its inspiration from the quaternionic formulation of quantum mechanics \cite{BF}.
Extending complex spectral theory to vector operators
has broad applications across several fields such as:

\textit{Quaternionic Quantum Mechanics.} Birkhoff and von Neumann \cite{BF} showed that the Schrödinger equation can be formulated in quaternionic settings, motivating the study of spectral theory for quaternionic operators \cite{adler}.

\textit{Vector Analysis.} The gradient operator with nonconstant coefficients, which models physical laws like heat propagation and mass diffusion, is expressed as $T=\sum_{i=1}^ne_ia_i(x)\partial_{x_i}$ with specific boundary conditions \cite{CMS24}.
In order to generate the fractional powers of these operators a powerful method is the
 $H^\infty$-functional calculus, see the books \cite{Haase,HYTONBOOK1,HYTONBOOK2} for the complex version.
The $H^\infty$-functional calculus for bi-sectorial Clifford operators was investigated in \cite{QUADRATIC,MS24} and this allows the definition of the fractional Fourier law.

\textit{Differential Geometry.} The Dirac operator on a Riemannian manifold $(M,g)$ generalizes the gradient using covariant derivatives, expressed as $\mathcal{D}=\sum_{i=1}^ne_i\nabla_{E_i}^\tau$, see \cite{DiracHarm}. Specific realizations include the Dirac operator in hyperbolic and spherical spaces \cite{DIRACHYPSPHE}.

\textit{Hypercomplex Analysis.} The standard Dirac operator $D$ and its conjugate $\overline{D}$ are central to this field \cite{DSS,DiracHarm}. Slice hyperholomorphic functions are characterized as the kernel of a global operator $G$ \cite{6Global}. Furthermore, the Dirac fine structures on the $S$-spectrum are defined by operators like $T_{\alpha,m}=D^\alpha(D\overline{D})^m$, leading to specific function spaces and functional calculi \cite{polypolyFS,CDPS1,Fivedim,Polyf1,Polyf2}, including $H^\infty$-versions \cite{MILANJPETER,MPS23}.

These fine structures connect slice hyperholomorphic and axially monogenic functions via the Fueter-Sce extension theorem, using powers of the Laplacian \cite{Fueter,TaoQian1,Sce,ColSabStrupSce}.

Distinct from the $S$-spectrum approach, the monogenic spectral theory, initiated by Jefferies, McIntosh, and Picton-Warlow \cite{JM}, is based on monogenic functions, which are functions in the kernel of the Dirac operator, and their Cauchy formula. This theory and its associated  with the monogenic spectrum and the monogenic functional calculus admits the $H^\infty$-version extensively covered in \cite{JBOOK,TAOBOOK}.

\section*{Declarations and statements}

\textbf{Data availability}. There are no data associated with the research in this paper.\\

\textbf{Conflict of interest}.  The author declares that he has no competing interests regarding the publication of this paper.\\ 

\textbf{Research funding.} This research did not receive funding.


\begin{thebibliography}{99}





\bibitem{ACS2016} D. Alpay, F. Colombo, I. Sabadini: {\it Slice hyperholomorphic Schur analysis}. Volume 256 of Operator Theory: Advances and Applications. Basel, Birkhäuser/Springer (2016).

 \bibitem{AlpayColSab2020} D. Alpay, F. Colombo, I. Sabadini: {\it Quaternionic de Branges spaces and characteristic operator function}. SpringerBriefs in Mathematics, Springer, Cham (2020).

\bibitem{adler} S. Adler: \textit{Quaternionic Quantum Mechanics and Quaternionic Quantum Fields}. Volume 88 of \textit{International Series of Monographs on Physics}. Oxford University Press, New York (1995).

\bibitem{DIRACHYPSPHE} I. Beschastnyi, F. Colombo, S. A. Lucas, I. Sabadini: \textit{The S-resolvent estimates for the Dirac operator on hyperbolic and spherical spaces}. arXiv:2504.12725.

\bibitem{BF} G. Birkhoff, J. von Neumann: \textit{The logic of quantum mechanics}. Ann. of Math. (2) \textbf{37}(4) (1936) 823--843.

 \bibitem{FJBOOK} F. Colombo, J. Gantner: \textit{Quaternionic closed operators, fractional powers and fractional diffusion processes}. Operator Theory: Advances and Applications \textbf{274}. Birkhäuser/Springer, Cham (2019).

     \bibitem{CGK} F. Colombo, J. Gantner, D.P. Kimsey: \textit{Spectral theory on the $S$-spectrum for quaternionic operators}. Volume 270 of Operator Theory: Advances and Applications, Birkhäuser/Springer, Cham (2018) ix+356.

\bibitem{polypolyFS} F. Colombo, A. De Martino, S. Pinton: \textit{Functions and operators of the polyharmonic and polyanalytic Clifford fine structures on the $S$-spectrum.} arXiv:2501.14716

\bibitem{CDPS1} F. Colombo, A. De Martino, S. Pinton, I. Sabadini: \textit{Axially harmonic functions and the harmonic functional calculus on the $S$-spectrum}, J. Geom. Anal. \textbf{33}(1) (2023) 54 pp.

\bibitem{Fivedim} F. Colombo, A. De Martino, S. Pinton, I. Sabadini: \textit{The fine structure of the spectral theory on the $S$-spectrum in dimension five}. J. Geom. Anal. \textbf{33} (9) (2023) 73 pp.

\bibitem{6Global} F. Colombo, O.J. Gonzalez-Cervantes, I. Sabadini: \textit{A nonconstant coefficients differential operator associated to slice monogenic functions}. Trans. Amer. Math. Soc., \textbf{365} (2013), no. 1, 303--318.

 \bibitem{CMS24} F. Colombo, F. Mantovani, P. Schlosser:
{\em Spectral properties of the gradient operator with nonconstant coefficients.}
 Anal. Math. Phys. 14 , no. 5, Paper No. 108, 31 p. (2024).
 
\bibitem{QUADRATIC} F. Colombo, F. Mantovani, P. Schlosser: 
{\em Quadratic estimates for the $H^\infty$-functional calculus of
bisectorial Clifford operators.}
The Journal of Geometric Analysis,
https://doi.org/10.1007/s12220-025-02282-z



\bibitem{MILANJPETER} F. Colombo, S. Pinton, P. Schlosser: {The $H^\infty$-functional calculi for the quaternionic fine structures of Dirac type}. Milan J. Math. \textbf{92}(1) (2024) 73--122.

\bibitem{ColSabStrupSce} F. Colombo, I. Sabadini, D.C. Struppa: \textit{Michele Sce's works in hypercomplex analysis---a translation with commentaries}. Birkhäuser/Springer, Cham (2020).

    \bibitem{ColomboSabadiniStruppa2011} F. Colombo, I. Sabadini, D.C. Struppa: {\it Noncommutative functional calculus}. Volume 289 of Progress in Mathematics. Birkhäuser/Springer Basel AG, Basel, Theory and applications of slice hyperholomorphic functions (2011).


\bibitem{DSS} R. Delanghe, F. Sommen, V. Sou\v cek: \textit{Clifford algebra and spinor-valued functions.} Mathematics and its Applications \textbf{53}, Kluwer Academic Publishers Group, Dordrecht (1992).

\bibitem{Polyf1} A. De Martino, S.~Pinton: {A polyanalytic functional calculus of order 2 on the $S$-spectrum}. Proc. Amer. Math. Soc. \textbf{151}(6) (2023) 2471--2488.

\bibitem{Polyf2} A. De Martino, S. Pinton: {Properties of a polyanalytic functional calculus on the $S$-spectrum}. Math. Nachr. \textbf{296}(11) (2023) 5190--5226.

\bibitem{MPS23} A. De Martino, S. Pinton, P. Schlosser: {The harmonic $H^\infty$-functional calculus based on the $S$-spectrum}. J. Spectr. Th. \textbf{14}(1) (2024) 121--162.

\bibitem{Fueter} R. Fueter: \textit{Die Funktionentheorie der Differentialgleichungen $\Delta u=0$ und $\Delta\Delta u=0$ mit vier reellen Variablen}. Comment. Math. Helv. \textbf{7}(1) (1934) 307--330.

\bibitem{Haase} M. Haase: \textit{The functional calculus for sectorial operators}. Operator Theory: Advances and Applications \textbf{169}, Birkhäuser Verlag, Basel (2006).

\bibitem{HYTONBOOK1} T. Hytönen, J. van Neerven, M. Veraar, L. Weis: \textit{Analysis in Banach spaces. Vol. I. Martingales and Littlewood-Paley theory}. Ergebnisse der Mathematik und ihrer Grenzgebiete, 3. Folge. A Series of Modern Surveys in Mathematics \textbf{63}. Springer, Cham (2016).

\bibitem{HYTONBOOK2} T. Hytönen, J. van Neerven, M. Veraar, L. Weis: \textit{Analysis in Banach spaces. Vol. II. Probabilistic methods and operator theory}. Ergebnisse der Mathematik und ihrer Grenzgebiete, 3. Folge. A Series of Modern Surveys in Mathematics \textbf{67}. Springer, Cham (2017).

\bibitem{DiracHarm} J.E. Gilbert, M.A.M. Murray: \textit{Clifford algebras and Dirac operators in harmonic analysis}. Cambridge Studies in Advanced Mathematics \textbf{26}. Cambridge University Press, Cambridge (1991) viii+334 p.

\bibitem{JBOOK} B. Jefferies: \textit{Spectral properties of noncommuting operators}. Lecture Notes in Mathematics \textbf{1843}, Springer-Verlag, Berlin (2004).

\bibitem{JM} B. Jefferies, A. McIntosh, J. Picton-Warlow: \textit{The monogenic functional calculus}. Studia Mathematica \textbf{136}(2) (1999) 99-119.


\bibitem{MS24} F. Mantovani, P. Schlosser: \textit{The $H^\infty$-functional calculus for bisectorial Clifford operators}. J. Spectr. Theory \textbf{15}, No. 2, 751-818 (2025).

\bibitem{TAOBOOK} T. Qian, P. Li: \textit{Singular integrals and Fourier theory on Lipschitz boundaries}. Science Press Beijing, Beijing, Springer, Singapore (2019).

\bibitem{TaoQian1} T. Qian: \textit{Generalization of Fueter's result to $\mathbb{R}^{n+1}$}. Atti Accad. Naz. Lincei. Cl. Sci. Fis. Mat. Natur. Rend. Lincei \textbf{8}(2) (1997) 111--117.

\bibitem{Sce} M. Sce: \textit{Osservazioni sulle serie di potenze nei moduli quadratici}. Atti Accad. Naz. Lincei Rend. Cl. Sci. Fis. Mat. Natur. \textbf{23}(8) (1957) 220--225.
\end{thebibliography}
\end{document}